\documentclass{amsart}
\usepackage{amsmath,amscd}
\usepackage{amsfonts}
\usepackage[all]{xy}
\usepackage{enumerate}
\numberwithin{equation}{subsection}
\theoremstyle{plain}
\newtheorem{thm}[subsection]{Theorem}
\newtheorem{prop}[subsection]{Proposition}
\newtheorem{lemma}[subsection]{Lemma}
\newtheorem{cor}[subsection]{Corollary}
\theoremstyle{definition}
\newtheorem{defn}[subsection]{Definition}
\newtheorem{notn}[subsection]{Notation}
\newtheorem{cont}[subsection]{Contents}
\newtheorem{ackn}[subsection]{Acknowledgement}
\theoremstyle{remark}
\newtheorem{rem}[subsection]{Remark}

\begin{document}
\title{On the nilpotent section conjecture for finite group actions on curves}
\author{Ambrus P\'al}
\date{August 1, 2013}
\address{Department of Mathematics, 180 Queen's Gate, Imperial College, London, SW7 2AZ, United Kingdom}
\email {a.pal@imperial.ac.uk}
\begin{abstract} We give a new, geometric proof of the section conjecture for fixed points of finite group actions on projective curves of positive genus defined over the field of complex numbers, as well as its natural nilpotent analogue. As a part of our investigations we give an explicit description of the abelianised section map for groups of prime order in this setting. We also show a version of the $2$-nilpotent section conjecture.
\end{abstract}
\maketitle
\pagestyle{myheadings}
\markboth{Ambrus P\'al}{The nilpotent section cojecture for finite group actions on curves}

\section{Introduction}

For every pro-group $\pi$ let $\pi=[\pi]_0>[\pi]_1>[\pi]_2>\ldots$ denote the lower central series of $\pi$, so $[\pi]_{n+1}=\overline{[[\pi]_n,\pi]}$ is the closure of the subgroup generated by commutators of elements of $[\pi]_n$ and $\pi$. Let $[\pi]_{\infty}$ be the closed subgroup:
$$[\pi]_{\infty}=\bigcap_{n=1}^{\infty}[\pi]_n.$$
For every short exact sequence $\mathbf E$ of pro-groups
\begin{equation}
\CD
1@>>>\pi@>>>\widetilde{\pi}@>>>G@>>>1
\endCD
\end{equation}
the conjugacy class of a section $s:G\rightarrow\widetilde{\pi}$
is the set of sections of the form
$$g\mapsto\gamma s(g)\gamma^{-1}$$
where $\gamma$ is any element of $\pi$. Taking the quotient by the closed  characteristic subgroup $[\pi]_n$ (where $n\in\mathbb N$ or $n=\infty$) gives the short exact sequence $\mathbf E_n$:
\begin{equation}
\CD
1@>>>\pi/[\pi]_n@>>>\widetilde\pi/[\pi]_n@>>>G@>>>1.
\endCD
\end{equation}
Let $S(\mathbf E)$ denote the set of conjugacy classes of sections of any short exact sequence $\mathbf E$ as above, and let $f_n(\mathbf E):S(\mathbf E)\rightarrow S(\mathbf E_n)$ denote the map which assigns to the conjugacy class of a section $s$ of (1.0.1) the conjugacy class of the composition of $s$ and the quotient map $\widetilde\pi\rightarrow\widetilde\pi/[\pi]_n$.

Now let $X$ be a smooth, projective, geometrically irreducible curve of positive genus over $\mathbb C$ and let $G$ be a finite group acting on $X$. Assume that every non-trivial subgroup of $G$ acts non-trivially on the algebraic curve $X$. Let $X^G$ denote the set of fixed points of $G$ on $X$. Fix a base point $x$ of $X$ and let
\begin{equation}
\CD
1@>>>\pi_1(X,x)@>>>\pi_1(X,G,x)@>>>G@>>>1
\endCD
\end{equation}
be the short exact sequence $\mathbf E(X,G)$ furnished by the action of $G$ on $X$ where $\pi_1(X,x)$ and $\pi_1(X,G,x)$ are the topological and the orbifold fundamental group of $X$ with base point $x$, respectively. Let $\Sigma$ be a set of primes and let $\mathcal C_{\Sigma}$ denote the complete class of finite groups whose order is only divisible by primes in $\Sigma$. Let $\pi^{\Sigma}_1(X,x)$ and $\pi^{\Sigma}_1(X,G,x)$ denote the pro-$\mathcal C_{\Sigma}$ completions of the groups $\pi_1(X,x)$ and $\pi_1(X,G,x)$, respectively. Assume that every prime which divides the order of $G$ is in $\Sigma$ and let
\begin{equation}
\CD
1@>>>\pi_1^{\Sigma}(X,x)
@>>>\pi_1^{\Sigma}(X,G,x)@>>>G@>>>1
\endCD
\end{equation}
be the short exact sequence $\mathbf E^{\Sigma}(X,G)$ that we get from (1.0.3) by taking the pro-$\mathcal C_{\Sigma}$ completions. There is a natural map:
$$s_{X,G}:X^G\longrightarrow S(\mathbf E(X,G))$$
(see Definition 2.1 for a detailed construction). Let $$c^{\Sigma}_{X,G}:S(\mathbf E(X,G))
\longrightarrow S(\mathbf E^{\Sigma}(X,G))$$
be the map furnished by the $\mathcal C_{\Sigma}$-completion functor and let $s^{n,\Sigma}_{X,G}$ denote the composition $f_n(\mathbf E^{\Sigma}(X,G))\circ c^{\Sigma}_{X,G}
\circ s_{X,G}$. One of the motivations to write this paper was to give a purely geometric proof to the following
\begin{thm} The map
$$s^{\infty,\Sigma}_{X,G}:X^G\longrightarrow S(\mathbf E^{\Sigma}(X,G)_{\infty})$$
is a bijection.
\end{thm}
It is easy to reduce the theorem above to the case when $G$ is a cyclic group of prime order using the argument involving Frobenius groups in my previous paper \cite{Pa}. In this case the results follows easily from Theorem 5.11 of \cite{Pa} on page 366, which is the section conjecture for the $\mathcal C_{\Sigma}$-completion of the fundamental group of $X$. In that paper I used topological methods to prove this result. However here we will give a purely geometric proof studying the action of $G$ on the Jacobian of $X$. Our central result is Theorem 4.7 below, which describes the map $s^{1,\Sigma}_{X,G}$ completely for groups of prime order. As a consequence of this result we will also get the following 
\begin{cor} Assume that $G$ has prime order. Then the map
$$s^{1,\Sigma}_{X,G}:X^G\longrightarrow S(\mathbf E^{\Sigma}(X,G)_1)$$
is injective if and only if $|X^G|\neq2$.
\end{cor}
In particular we will see that the map $s^{1,\Sigma}_{X,G}$ is neither injective nor surjective in general. It is natural to ask whether we can reconstruct the set $X^G$ using the maps $s^{1,\Sigma}_{X,G}$ and $s^{2,\Sigma}_{X,G}$, similarly to the main result of \cite{Wi}. Our theorem says that this is true when $|X^G|\geq3$ and $G$ is a $2$-group. Let
$$d^{n,\Sigma}_{X,G}:S(\mathbf E^{\Sigma}(X,G)_n)
\longrightarrow S(\mathbf E^{\Sigma}(X,G)_{n-1})$$
be the map which assigns to the conjugacy class of a section
$$s:G\longrightarrow
\pi_1^{\Sigma}(X,G,x)/[\pi_1^{\Sigma}(X,x)]_n$$
the conjugacy class of the composition of $s$ and the quotient map $$\pi_1^{\Sigma}(X,G,x)/[\pi_1^{\Sigma}(X,x)]_n\longrightarrow
\pi_1^{\Sigma}(X,G,x)/[\pi_1^{\Sigma}(X,x)]_{n-1}.$$
Let $\overline S(\mathbf E^{\Sigma}(X,G)_{n})$ denote the image of $d^{n,\Sigma}_{X,G}$. Our last main result is the following
\begin{thm} The image of the map
$$s^{1,\Sigma}_{X,G}:X^G\longrightarrow S(\mathbf E^{\Sigma}(X,G)_1)$$
is $\overline S(\mathbf E^{\Sigma}(X,G)_2)$ when $G$ is a group of order two.
\end{thm}
This result is closely analogous to Wickelgren's Theorem 1.1 of \cite{Wi} and its proof is also very similar. Note that for every subgroup $H\leq G$ we have $|X^H|\geq|X^G|$. Therefore by repeating the inductive proof of Proposition 2.5 in section 3, the theorem above combined with Corollary 1.2 implies the following
\begin{cor} Assume that $G$ is a $2$-group and $|X^G|\geq3$. Then the map
$$s^{1,\Sigma}_{X,G}:X^G\longrightarrow \overline S(\mathbf E^{\Sigma}(X,G)_2)$$
is a bijection.\qed
\end{cor}
\begin{cont} In the next chapter we give the definition of the section map and we deduce Theorem 1.1 from the special case when $G$ has prime order. In the third section we prove the section conjecture for abelian varieties using geometric methods. We study the induced action of $G$ on the Picard scheme when $G$ has prime order, and as an application we compute the map $s^{1,\Sigma}_{X,G}$ rather explicitly in the fourth section. The main result of this section, Theorem 4.7, looks classical, but we could not find a reference. We complete the proof of Theorem 1.1 in the fifth section using Tamagawa's trick and some simple group-theoretical arguments. We use cohomological obstruction theory to show Theorem 1.3 in the last section.
\end{cont}
\begin{ackn} The author was partially supported by the EPSRC grant P36794. I also wish to thank the referee for many useful comments and pointing out some errors in the original version.
\end{ackn}

\section{Reduction to the case of a group of prime order}

\begin{defn} In this paper by a topological space we always mean a Hausdorff topological space. For every topological space $Y$ let Aut$(Y)$ denote the group of its self-homeomorphisms. Let $X$ be a connected locally contractible topological space. Let $G$ be a finite group and assume that a continuous left action $h:G\times X\rightarrow X$ of $G$ on $X$ is given. Let $x$ be a point of $X$ and let $(\widetilde X,\widetilde x)$ be the universal cover of the pointed space $(X,x)$. Let $\pi:\widetilde X\rightarrow X$ be the covering map (with the property $\pi(\widetilde x)=x$). We define the
orbifold fundamental group $\pi_1(X,G,x)$ to be the following subgroup of Aut$(\widetilde X)\times G$:
$$\pi_1(X,G,x)=\{(\phi,g)\in\text{Aut}(\widetilde X)\times G|
\pi\circ\phi=h(g)\circ\pi\}.$$
By definition the group $\pi_1(X,G,x)$ acts on $\widetilde X$. Moreover the fundamental group $\pi_1(X,x)$, considered as the group of deck transformations of the cover $\pi:\widetilde X\rightarrow X$, is a subgroup of $\pi_1(X,G,x)$ via the map $\phi\mapsto(\phi,1)$. Because every continuous map $X\rightarrow X$ can be lifted by the universal property of the covering map $\pi$ we have an exact sequence $\mathbf E(X,G)$:
\begin{equation}
\CD1@>>>\pi_1(X,x)@>>>\pi_1(X,G,x)@>\psi_{X,G}>>G@>>>1.\endCD
\end{equation}
Let $X^G$ denote the set of fixed points of the action of $G$ on $X$. Let $y\in X^G$ and choose a $\widetilde y\in\widetilde X$ such that $\pi(\widetilde y)=y$. For every $g\in G$ there is a unique $h_{\widetilde y}(g)\in\pi_1(X,G,x)$ which fixes $\widetilde y$ and $\psi_{X,G}(h_{\widetilde y}(g))=g$. The map $h_{\widetilde y}:G\rightarrow\pi_1(X,G,x)$ is a homomorphism which is also a section of the exact sequence (2.1.1) by construction. Let $\widetilde y'\in\widetilde X$ be another point such that $\pi(\widetilde y')=y$. Then there is a unique $\gamma\in\pi_1(X,x)$ such that $\widetilde y'=\gamma(\widetilde y)$. Clearly $h_{\widetilde y'}=\gamma^{-1}h_{\widetilde y}\gamma$. By the above the conjugacy class of $h_{\widetilde y}$ does not depend on the choice of the lift $\widetilde y$. Hence we get a well-defined map:
$$s_{X,G}:X^G\longrightarrow S(\mathbf E(X,G))$$
which associates to $y$ the conjugacy class of $h_{\widetilde y}$. 
\end{defn}
\begin{lemma} The map $s_{X,G}$ is constant on the path-connected components of $X^G$.
\end{lemma}
\begin{proof} This is Lemma 2.2 of \cite{Pa} on page 355.
\end{proof}
\begin{defn} Let $\Sigma$ be a set of primes and let $\mathcal C_{\Sigma}$ denote the complete class introduced in the introduction. For every group $H$ let $H^{\Sigma}$ denote the $\mathcal C_{\Sigma}$-completion of $H$ (for the definition of the latter see page 26 of \cite{AM}). Moreover let $\pi^{\Sigma}_1(X,x)$ and $\pi^{\Sigma}_1(X,G,x)$ denote the profinite groups $\pi_1(X,x)^{\Sigma}$ and $\pi_1(X,G,x)^{\Sigma}$, respectively. By applying the $\mathcal C_{\Sigma}$-completion functor to (2.1.1) we get the following exact sequence $\mathbf E^{\Sigma}(X,G)$:
$$\CD1@>>>\pi^{\Sigma}_1(X,x)@>>>
\pi^{\Sigma}_1(X,G,x)@>>>G@>>>1\endCD$$ 
when $G$ is in $\mathcal C_{\Sigma}$. The $\mathcal C_{\Sigma}$-completion induces a natural map
$$c^{\Sigma}_{X,G}:S(\mathbf E(X,G))
\rightarrow S(\mathbf E^{\Sigma}(X,G)).$$
For every topological space $M$ let $\pi_0(M)$ denote the set of connected components of $M$. By the above when $X^G$ is locally path-connected for every $n\in\mathbb N$ there is a well-defined map:
$$s^{n,\Sigma}_{X,G}:\pi_0(X^G)\longrightarrow S(\mathbf E^{\Sigma}(X,G)_n)$$
which sends every connected component $C$ to the class $f_n(\mathbf E^{\Sigma}(X,G))(c^{\Sigma}_{X,G}(s_{X,G}(y)))$ where $y\in C$ is arbitrary.
\end{defn}
\begin{notn} For every algebraic variety $X$ over $\Bbb C$ let the symbol $X$ also denote the set of $\mathbb C$-valued points of $X$ equipped with the analytic topology by slight abuse of notation. Let $X$ be as above and assume that $G$ is a finite group acting on $X$. Let $X^G$ denote the variety of fixed points of the action of $G$ on $X$. Since algebraic varieties are locally path-connected for the analytic topology we have a well-defined map:
$$s^{n,\Sigma}_{X,G}:\pi_0(X^G)\longrightarrow S(\mathbf E^{\Sigma}(X,G)_n)$$
by the above. 
\end{notn}
Let $X$ again be a smooth, projective, geometrically irreducible curve of positive genus over $\mathbb C$ and let $G$ be a finite group acting on $X$. Assume that every non-trivial subgroup of $G$ acts non-trivially on the algebraic curve $X$. In this case every connected component of $X^G$ consists of one point. We will identify the sets $X^G$ and $\pi_0(X^G)$ in all that follows. 
\begin{prop} Assume that Theorem 1.1 holds when $G$ is a cyclic group of prime order. Then it holds in general.
\end{prop}
\begin{proof} Let $H$ be a subgroup of $G$. Then there is a commutative diagram:
$$
\CD1\rightarrow\pi^{\Sigma}_1(X,x)/[\pi^{\Sigma}_1(X,x)]_{\infty}@>>>
\pi^{\Sigma}_1(X,H,x)/[\pi^{\Sigma}_1(X,x)]_{\infty}@>>>H\rightarrow1\\
@VVV@VVV\!\!\!\!\!\!\!\!\!\!\!\!
@VVV\\
1\rightarrow\pi^{\Sigma}_1(X,x)/[\pi^{\Sigma}_1(X,x)]_{\infty}@>>>
\pi^{\Sigma}_1(X,G,x)/[\pi^{\Sigma}_1(X,x)]_{\infty}@>>>G\rightarrow1
\endCD
$$
such that the square on the right hand side is Cartesian because $H$ belongs to the class $\mathcal C_{\Sigma}$. The restriction of  sections
$$G\longrightarrow\pi^{\Sigma}_1(X,G,x)/[\pi^{\Sigma}_1(X,x)]_{\infty}$$
onto $H$ are sections
$$H\longrightarrow\pi^{\Sigma}_1(X,H,x)/[\pi^{\Sigma}_1(X,x)]_{\infty},$$
and hence this operation furnishes a map $r_{X,G,H}:S(\mathbf E^{\Sigma}(X,G)_{\infty})\rightarrow S(\mathbf
E^{\Sigma}(X,H)_{\infty})$ which makes diagram:
$$\CD X^G@>>s^{\infty,\Sigma}_{X,G}>
S(\mathbf E^{\Sigma}(X,G)_{\infty})\\
@VVV@VV{r_{X,G,H}}V\\
X^H@>s^{\infty,\Sigma}_{X,H}>>
S(\mathbf E^{\Sigma}(X,H)_{\infty})
\endCD$$
commutative where the left vertical map is the inclusion. Suppose now that  $H$ has prime order. Then the map $s^{\infty,\Sigma}_{X,H}$ is injective by assumption. Since the inclusion $X^G\subseteq X^H$ is also injective, the map $s^{\infty,\Sigma}_{X,G}$ is injective.

We are going to prove that $s^{\infty,\Sigma}_{X,G}$ is bijective by induction on the order of $G$, that is, we will assume that we proved the claim for every finite group $H$ whose order is strictly less than the order of $G$. By the above we only have to prove surjectivity. Suppose first that $G$ has a proper normal subgroup $H$. Let $$\phi:G\longrightarrow\pi^{\Sigma}_1(X,G,x)/[\pi^{\Sigma}_1(X,x)]_{\infty}$$
be a section of the lower short exact sequence of the first diagram in this proof. The restriction $\phi|_H$ of $\phi$ to $H$ is a section of the upper short exact sequence of the same diagram, and hence there is a unique fixed point $y\in X^H$ corresponding to the conjugacy class of $\phi|_H$ by assumption. Because $H$ is a normal subgroup for every $g\in G$ the point $g(y)$ is also fixed by $H$. Moreover $s^{\infty,\Sigma}_{X,H}(g(y))$ is the conjugacy class of the sections given by the rule
$$\psi(h)=\widetilde g\phi(g^{-1}hg)\widetilde g^{-1}\quad(\forall h\in H)$$
where $\widetilde g\in \pi^{\Sigma}_1(X,G,x)/[\pi^{\Sigma}_1(X,x)]_{\infty}$ is any lift of $g$. By choosing $\widetilde g=\phi(g)$ we get that $$\psi(h)=\phi(g)\phi(g^{-1}hg)\phi(g)^{-1}=
\phi(g)\phi(g)^{-1}\phi(h)\phi(g)\phi(g)^{-1}=\phi(h)$$
for every $g$ and $h$, so all these sections are conjugate. So by the injectivity of $s^{\infty,\Sigma}_{X,H}$ the point $y$ is fixed by $G$ and hence $s^{\infty,\Sigma}_{X,G}$ is a bijection.

Assume now that $G$ is a simple group. Then either $G$ is isomorphic to a group of prime order or it is non-abelian. When $G$ has prime order we know the claim by assumption. Hence we may assume that $G$ is not commutative. If the set $S(\mathbf E^{\Sigma}(X,G)_{\infty})$ is empty there is nothing to prove. Assume now that it is non-empty; then the same also holds for every proper subgroup of $G$, hence the latter have a fixed point by the induction hypothesis. In particular every proper subgroup of $G$ is cyclic, since the stabiliser of a point under an algebraic action of a finite group on a curve is necessarily cyclic. At this point it is possible to derive a contradiction by purely group-theoretical means: let $H$ be a maximal proper subgroup of $G$. Fix an element $x\in G-H$ and let $J$ be the subgroup $H\cap xHx^{-1}$. Since $J\subseteq H$, it is cyclic, so it is generated by an element $y\in H$. By definition $y=xzx^{-1}$ for some $z\in H$. The order of $z$ is the same as the order of $xzx^{-1}=y$, and hence $z=y^i$ for some integer $i$. Therefore $x^{-1}yx=z=y^i\in J$ and hence $x^{-1}Jx\subseteq J$. We get that $J$ is normalized by $x$ and it is also normalized by $H$ because the latter is commutative. Hence $J$ is a normal subgroup in $G$ since $x$ and $H$ generate $G$. Therefore it is the trivial subgroup. So $G$ is a Frobenius group. But Frobenius groups are never simple because of the existence of the Frobenius complement.
\end{proof}

\section{The section conjecture for abelian varieties}

\begin{notn} Let $\mathbf X=(X,x)$ be a connected locally contractible pointed topological space and let $\Sigma$ be a non-empty set of prime numbers. Let $I(\mathbf X,\mathcal C_{\Sigma})$ be the set of all open normal subgroups $K$ of $\pi_1^{\Sigma}(X,x)$. For every $K\in I(\mathbf X,\mathcal C_{\Sigma})$ let $\pi_K:(X_K,x_K)\rightarrow(X,x)$ be the covering map of pointed topological spaces corresponding to pre-image of finite index subgroup $K$ with respect to the completion map $\pi_1(X,x)\rightarrow \pi_1^{\Sigma}(X,x)$. Assume that a non-trivial continuous left action $h:G\times X\rightarrow X$ of $G$ on $X$ is given such that $X^G$ is locally path-connected and $G\in\mathcal C_{\Sigma}$. Let $I(\mathbf X,\mathcal C_{\Sigma},G)$ denote the set of open subgroups $K\in I(\mathbf X,\mathcal C_{\Sigma})$ which are also normal as a subgroup of the larger group $\pi^{\Sigma}_1(X,G,x)$. 
\end{notn}
\begin{defn} Let $j:\pi_1(X,G,x)\rightarrow \pi_1^{\Sigma}(X,G,x)$ denote the completion map. Let $s:G\rightarrow\pi^{\Sigma}_1(X,G,x)$ be a section of (1.0.4), and for every $K\in I(\mathbf X,\mathcal C_{\Sigma},G)$ and for every $g\in G$ pick an element $\widetilde g_K\in\pi_1(X,G,x)$ such that $j(\widetilde g_K)$ and $s(g)$ are in the same right $K$-coset. By our assumptions on $K$ the element $\widetilde g_K$ normalizes the subgroup $K\cap\pi_1(X,x)$ hence there is a unique homeomorphism $h^K_s(g):X_K\rightarrow X_K$ such that $\pi^K\circ\widetilde g_K=h^K_s(g)\circ\pi^K$. Moreover $h^K_s(g)$ is independent of the choice of $\widetilde g_K$ and the map $g\mapsto h^K_s(g)$ defines a continuous left-action of $G$ on $X_K$ which will be denoted simply by $h^K_s$. With respect to this action the map $\pi_K$ is $G$-equivariant.
\end{defn}
Let
$$s^{\Sigma}_{X,G}:\pi_0(X^G)
\longrightarrow S(\mathbf E^{\Sigma}(X,G))$$
denote the composition $c^{\Sigma}_{X,G}\circ s_{X,G}$. The next claim is very well-known, and it is usually called Tamagawa's trick. (See part $(iv)$ of Proposition 2.8 of \cite{Ta} on page 152.)
\begin{lemma} Assume that $X$ is compact. Then the conjugacy class of a section $s$ of (1.0.4) lies in the image of the map $s^{\Sigma}_{X,G}$ if and only if $X^G_K$ is non-empty with respect to the action $h^K_s$ for every $K\in I(\mathbf X,\mathcal C_{\Sigma},G)$.\qed
\end{lemma}
\begin{defn} Let $\Gamma$ be a group, let $\Sigma$ be a set of primes and let $p$ be an element of $\Sigma$. We will say that a group $\Gamma$ is $(\Sigma,p)$-good if the homomorphism of cohomology groups $H^n(\Gamma^{\Sigma}, M)\rightarrow H^n(\Gamma, M)$ induced by the natural homomorphism $\Gamma\rightarrow\Gamma^{\Sigma}$ is an isomorphism for every finite $p$-torsion $\Gamma^{\Sigma}$-module $M$. By Example 5.4 of \cite{Pa} on page 363 every finitely generated free abelian group is $(\Sigma,p)$-good for every $\Sigma$ and $p$ as above.
\end{defn} 
\begin{lemma} Let $\Gamma$ be a group. Assume that it has a normal subgroup $\Delta$ such that the quotient group $\Gamma/\Delta$ is in $\mathcal C_{\Sigma}$ and $\Delta$ is $(\Sigma,p)$-good and finitely generated. Then $\Gamma$ is also $(\Sigma,p)$-good.
\end{lemma}
\begin{proof} This claim is just a minor variant of part $(b)$ of Exercise 2 of \cite{Se} on page 13. Because $\Delta$ is finitely generated, every normal subgroup $N\triangleleft\Delta$ such that $\Delta/N\in\mathcal C_{\Sigma}$ contains a characteristic subgroup $R\triangleleft\Delta$ such that $\Delta/R\in\mathcal C_{\Sigma}$, namely, the intersection of the kernels of all homomorphisms $\Delta\rightarrow\Delta/N$. Therefore the map $\Delta^{\Sigma}\rightarrow
\Gamma^{\Sigma}$ induced by the inclusion map $\Delta\rightarrow\Gamma$ is injective and the quotient groups $\Gamma^{\Sigma}/\Delta^{\Sigma}$ and $\Gamma/\Delta$ are isomorphic. For every pro-group $R$, closed subgroup $S\leq R$ and $S$-module $M$ let $\textrm{Ind}^S_R(M)$ denote the $R$-module induced from $M$. By the above for every finite $p$-torsion $\Delta^{\Sigma}$-module $M$, the induced modules  $\textrm{Ind}^{\Delta}_{\Gamma}(M)$ and
$\textrm{Ind}^{\Delta^{\Sigma}}_{\Gamma^{\Sigma}}(M)$ are the same, and we have a commutative diagram:
$$\CD
H^n(\Gamma^{\Sigma},M)@>>>H^n(\Gamma,M)\\
@VVV@VVV\\
H^n(\Delta^{\Sigma},\textrm{Ind}^{\Delta^{\Sigma}}_{\Gamma^{\Sigma}}(M))
@>>>
H^n(\Delta,\textrm{Ind}^{\Delta}_{\Gamma}(M))
\endCD
$$
where the vertical maps are those defined on page 11 of \cite{Se}. By Shapiro's lemma the vertical maps are isomorphisms. By assumption the lower horizontal map is also an isomorphism, so the claim is now clear.
\end{proof}
In the rest of the paper we will assume that $G$ is a cyclic group of prime order $p$. 
\begin{thm} Let $A$ be an abelian variety over $\mathbb C$ equipped with an algebraic action of $G$. Then the map
$$s^{\Sigma}_{A,G}:\pi_0(A^G)\longrightarrow S(\mathbf E^{\Sigma}(A,G))$$
is a bijection.
\end{thm}
\begin{proof} This claim is just a special case of Theorem 1.3 of \cite{Pa} on page 354. However we are going to give a geometric proof in the spirit of this article. First we are going to prove that the map $s^{\Sigma}_{A,G}$ is surjective. Since every finite \'etale cover of an abelian variety is also an abelian variety, it will be enough to show that $A^G$ is non-empty when $S(\mathbf E^{\Sigma}(A,G))$ is non-empty, by Tamagawa's trick. Assume that this is not the case; then the action of $G$ on $A$ is free.

Let $A/G$ denote the quotient of $A$ by this free action. This is a smooth, projective variety, so it  has finite cohomological dimension. The quotient map
$A\rightarrow A/G$ is a finite topological cover and hence $A/G$ is an Eilenberg-MacLane space. Its fundamental group is the extension of $G$ by the fundamental group of $A$. Since the latter is a finitely generated abelian group, from Lemma 3.5 we get that
$$H^n(A/G,\mathbb F_p)\cong H^n(\pi_1(A/G,y),\mathbb F_p)$$
is isomorphic to $H^n(\pi^{\Sigma}_1(A/G,y),\mathbb F_p)$ for every $n\in\mathbb N$, where $y$ is any point of $A/G$. So $\pi^{\Sigma}_1(A/G,y)$ has finite mod-$p$ cohomological dimension. But by assumption $\pi^{\Sigma}_1(A/G,y)$ has a closed subgroup isomorphic to the finite group $G$. As the cohomological dimension of $G$ is infinite, this is a contradiction.

Next we prove that the map $s^{\Sigma}_{A,G}$ is injective. We may assume without the loss of generality that $A^G$ is non-empty, so we may also suppose that the identity element of $A$ is fixed by $G$. In this case the action of $G$ respects the group structure of $A$, and hence $A^G$ is a smooth projective algebraic group over $\mathbb C$. Since the connected component $B$ of the identity in $A^G$ is connected by definition, it is an abelian variety. Clearly $\pi_0(A^G)=A^G/B$ so the the set $\pi_0(A^G)$ has the structure of a finite abelian group. For every $n\in\mathbb N$ and abelian variety $C$ let $C[n]$ denote the $n$-torsion of $C$. Let $n\in\mathbb N$ be divisible by the order of $\pi_0(A^G)$. Let 
$$\CD0@>>>A[n]@>>>A@>{x\mapsto x^n}>>A@>>>0\endCD$$
be the Kummer short exact sequence and let 
$$\CD A^G/nA^G@>{\delta_n}>> H^1(G,A[n])\endCD$$
be the associated coboundary map. The multiplication by $n$ map is surjective on the abelian variety $B$, so we get that $A^G/nA^G=\pi_0(A^G)$ and hence $\delta_n$ induces an injection $\pi_0(A^G)\rightarrow H^1(G,A[n])$. The group $H^1(G,A[n])$ is $p$-torsion so we get that $\pi_0(A^G)$ is $p$-torsion, too.

Let $x,y\in A^G$ be two fixed points lying on two different components $X$ and $Y$ of $A^G$, respectively. We want to show that $s^{\Sigma}_{A,G}(X)$ and $s^{\Sigma}_{A,G}(Y)$ are different. We may assume without the loss of generality that $x$ is the identity element of $A$. Assume that $s=s^{\Sigma}_{A,G}(X)$ and $s^{\Sigma}_{A,G}(Y)$ are the same. For every $K\in I((A,x),\mathcal C_{\Sigma},G)$ let $\pi_K:(A_K,x_K)\rightarrow (A,x)$ be the covering map introduced in Definition 3.2. We may equip $A_K$ uniquely with the structure of an abelian variety such that $x_K$ is the identity of $A_K$ and $\pi_K$ is an isogeny of abelian varieties. Note that with respect to the action $h^K_s$ on $A_K$ the isogeny $\pi_K$ is $G$-equivariant and by assumption there is a $y_K\in A_K$ fixed by $G$ which maps to $y$ with respect to $\pi_K$ for every such $K$. 

Let $\iota:C_1\rightarrow C_2$ be an isogeny of abelian varieties over $\mathbb C$. Then every automorphism of the group scheme $C_2$ has at most one lift to an automorphism of $C_1$ with respect to $\iota$, since the map $\iota$ induces an isomorphism on the tangent spaces of the identities of the abelian varieties $C_1$ and $C_2$. In particular the the action of $G$ on $A_{n\pi_1^{\Sigma}(A,x)}=A$ (where $x_K=x$) must be the given action on $A$ for every $n\in\mathbb N$ whose prime divisors are all in $\Sigma$. In particular for every such $n$ the point $y$ lies in $nA^G$. By the above this implies that $y$ lies on the same component as $x$ which is a contradiction.
\end{proof}
Since the fundamental group of abelian varieties over $\mathbb C$ are commutative, the theorem above has the following immediate
\begin{cor} Let $A$ be an abelian variety over $\mathbb C$ equipped with an action of $G$. Then the map
$$s^{1,\Sigma}_{A,G}:\pi_0(A^G)\longrightarrow S(\mathbf E^{\Sigma}(A,G)_1)$$
is a bijection.\qed
\end{cor}

\section{The abelian section map for groups of prime order}

\begin{defn} Assume now that $X$ is a smooth, projective, geometrically irreducible curve of positive genus over $\mathbb C$ and $G$ is a group of prime order $p$ acting on $X$ non-trivially. For every smooth projective connected curve $Y$ over $\mathbb C$ let $\textrm{Pic}(Y)$ and $\textrm{Pic}^n(Y)$ denote the Picard scheme of $Y$ and the degree $n$ component of the Picard scheme of $Y$, respectively. By functoriality the variety $\textrm{Pic}^n(X)$ is equipped with an action of $G$. Let $X/G$ denote the quotient of $X$ with respect to the action of $G$. Then $X/G$ is also a smooth projective connected curve over $\mathbb C$. Let $\pi:X\rightarrow X/G$ be the quotient map. 
\end{defn}
\begin{prop} The connected component of the identity in $\textrm{\rm Pic}^0(X)^G$ is the image of $\textrm{\rm Pic}^0(X/G)$ under the map $\pi^*:\textrm{\rm Pic}^0(X/G)\rightarrow\textrm{\rm Pic}^0(X)$ induced by Picard functoriality.
\end{prop}
\begin{proof} The action of $G$ respects the group structure of $\textrm{\rm Pic}^0(X)$, and hence $\textrm{\rm Pic}^0(X)^G$ is a smooth projective algebraic group over $\mathbb C$. Therefore the connected component $A$ of the identity in $\textrm{\rm Pic}^0(X)^G$ is an abelian variety. The image $B$ of $\textrm{\rm Pic}^0(X/G)$ with respect to $\pi^*$ is also an abelian variety which is contained by $A$. 

For every abelian variety $C$ over $\mathbb C$ let $T(C)$ denote tangent space of $C$ at the identity. By the above it will be enough to show that $T(A)\subseteq T(\textrm{Pic}^0(X))$ and $T(B)\subseteq T(\textrm{Pic}^0(X))$ are equal. There is a natural, and hence $G$-equivariant isomorphism between $T(\textrm{Pic}^0(X))$ and $H^1(X,\mathcal O_X)$. Under this isomorphism $T(A)$ corresponds to $H^1(X,\mathcal O_X)^G$ with respect to the induced action of $G$ and $T(B)$ corresponds to $\pi^*(H^1(X/G,\mathcal O_{X/G}))$. 
Let $\mathcal U=\{U_i\}_{i\in I}$ be a finite, Zariski-open covering by affine subschemas of $X/G$ and equip $I$ with a well-ordering. Then $\mathcal V=\{\pi^{-1}(U_i)\}_{i\in I}$ is also a finite, Zariski-open covering by affine subschemas of $X$. Therefore the \v Cech complex $C^*(\mathcal V,\mathcal O_X)$ computes the cohomology of $\mathcal O_X$ by Theorem III.4.5 of \cite{Ha}. Because the map $\pi$ is flat and affine, the subcomplex $C^*(\mathcal V,\mathcal O_X)^G$ of $G$-invariants is isomorphic to the \v Cech complex $C^*(\mathcal U,\mathcal O_{X/G})$ by Grothendieck's decent. But the functor of $G$-invariants is exact on the category of $\mathbb C[G]$-modules, so the $j$-th cohomology of $C^*(\mathcal V,\mathcal O_X)^G$ is $H^j(X,\mathcal O_X)^G$, that is, we get that
$H^j(X,\mathcal O_X)^G\cong H^j(X/G,\mathcal O_{X/G})$ for every $j$. The claim is now clear.
\end{proof}
\begin{prop} Let $\mathcal L\in\textrm{\rm Pic}^n(X)^G$. Then there is a $G$-invariant divisor $D$ on $X$ of degree $n$ such that the linear equivalence class of $D$ is $\mathcal L$. 
\end{prop}
\begin{proof} We may assume without the loss of generality that $\mathcal L$ is very ample by adding to $\mathcal L$ the linear equivalence class of the divisor $m\sum_{g\in G}g(y)$ for some $\mathbb C$-valued point $y$ and for a sufficiently large $m$. Let $g\in G$ be a generator and choose an isomorphism $\iota:\mathcal L\rightarrow g_*(\mathcal L)$. Note that $\pi_*$ furnishes an isomorphism $H^0(X/G,\pi_*(\mathcal L))\cong H^0(X,\mathcal L)$ and hence $H^0(X/G,\pi_*(\mathcal L))$ has positive dimension. Therefore there is a non-zero $s\in H^0(X/H,\pi_*(\mathcal L))$ which is an eigenvector for the automorphism:
$$\pi_*(\iota):\pi_*(\mathcal L)\longrightarrow\pi_*(g_*(\mathcal L))=
\pi_*(\mathcal L)$$
that we get by pushing forward $\iota$ with respect to $\pi$. The divisor of the image of $s$ with respect to the inverse of $\pi_*$ is a $G$-invariant divisor on $X$ whose linear equivalence class is $\mathcal L$. 
\end{proof} 
\begin{prop} Assume that $S(\mathbf E^{\Sigma}(X,G)_1)$ is non-empty. Then $X^G$ is non-empty too.
\end{prop}
\begin{proof} Let $[\cdot]:X\rightarrow\textrm{\rm Pic}^1(X)$ be the morphism which maps every point $x$ to its linear equivalence class $[x]$. This map is $G$-equivariant, so it induces a commutative diagram:
$$
\CD1\rightarrow\pi^{\Sigma}_1(X,x)/[\pi^{\Sigma}_1(X,x)]_1@>>>
\pi^{\Sigma}_1(X,G,x)/[\pi^{\Sigma}_1(X,x)]_1@>>>G\rightarrow1\\
@VVV@VVV\!\!\!\!\!\!\!\!\!\!\!\!
@|\\
1\rightarrow\frac{\displaystyle{\pi^{\Sigma}_1(\textrm{Pic}^1(X),[x])}}{\displaystyle{[\pi^{\Sigma}_1(\textrm{Pic}^1(X),[x])]_1}}@>>>
\frac{\displaystyle{\pi^{\Sigma}_1(\textrm{Pic}^1(X),G,[x])}}
{\displaystyle{[\pi^{\Sigma}_1(\textrm{Pic}^1(X),[x])]_1}}
@>>>G\rightarrow1,\endCD
$$
and hence we get that there is a commutative diagram:
$$
\CD X^G=\pi_0(X^G)@>s^{1,\Sigma}_{X,G}>>S(\mathbf E^{\Sigma}(X,G)_1)\\
@V\pi_0([\cdot]|_{X^G})VV@VVV\\
\pi_0(\textrm{Pic}^1(X)^G)@>s^{1,\Sigma}_{\textrm{Pic}^1(X),G}>>
S(\mathbf E^{\Sigma}(\textrm{Pic}^1(X),G)_1).
\endCD
$$
In particular we get that if $S(\mathbf E^{\Sigma}(X,G)_1)$ is non-empty then $S(\mathbf E^{\Sigma}(\textrm{Pic}^1(X),G)_1)$ is non-empty, too. Therefore $\textrm{\rm Pic}^1(X)^G$ is non-empty by Corollary 3.7, and hence by Proposition 4.3 there is a $G$-invariant divisor $D$ of degree one on $X$. Then $D$ can be written in the form:
$$\sum_{y\in X^G}n_y[y]+\sum_{z\in X-X^G}n_z(\sum_{g\in G}[g(z)]),$$
where $n_y,n_z\in\mathbb N$ and all but finitely many $n_z$-s are zero. Since the degree of the second summand is divisible by $p$ we get that first summand is non-zero. In particular $X^G$ is non-empty.
\end{proof} 
For every finite set $S$ let $\mathbb F_p\textrm{Div}^0(S)$ denote the group of formal linear combinations of elements of $S$ of the form:
$$\{\sum_{y\in S}n_yy|n_y\in\mathbb F_p,\ \sum_{y\in S}n_y=0\}.$$
For every smooth projective connected curve $Y$ over $\mathbb C$ let $F(Y)$ denote the function field of $Y$. By Kummer theory there is an $\alpha\in F(X)$ such that as an extension of $F(X/G)$ we have $F(X)=F(X/G)(\alpha)$ and $\alpha^p\in F(X/G)$.
\begin{lemma} Let $\alpha\in F(X)$ be such that $F(X)=F(X/G)(\alpha)$ and $\alpha^p\in F(X/G)$. Let $r(\alpha)$ denote the part of the divisor of $\alpha$ supported on $X^G$ mod $p$. Then the support of $r(\alpha)$ is $X^G$ and the $\mathbb F_p$-module generated by $r(\alpha)$ in $\mathbb F_p\textrm{\rm Div}^0(X^G)$ is independent of the choice of $\alpha$. 
\end{lemma}
\begin{proof} For every point $y\in X$ the value of the normalised valuation corresponding to $y$ taken on $\alpha$ is divisible by $p$ if an only if $\pi$ is unramified at $y$. The first claim is now clear. Let $\beta\in F(X)$ be another element such that $F(X)=F(X/G)(\beta)$ and $\beta^p\in F(X/G)$. By Galois theory there are primitive $p$-th roots of unity $\eta,\zeta\in\mathbb C^*$ such that $g(\alpha)=\eta\alpha$ and $g(\beta)=\zeta\beta$ where $g$ is a generator of the Galois group $G$ of the extension $F(X)/F(X/G)$. There is a natural number $i$ not divisible by $p$ such that $\zeta=\eta^i$ and hence $\beta/\alpha^i\in F(X/G)$ by Galois theory. Therefore the part of the divisor of $\beta$ supported on $X^G$ mod $p$ is $i$ times that of $\alpha$ mod $p$. The claim is now clear.
\end{proof}
Let $R_{X,G}\leq\mathbb F_p\textrm{Div}^0(X^G)$ denote the $\mathbb F_p$-module generated by the part of the divisor of $\alpha$ supported on $X^G$ mod $p$ for any $\alpha$ as above. 
\begin{notn} For every pro-group $\pi$ as in the introduction let $\pi^{ab}$ denote the quotient $\pi/[\pi]_1$. Let $r:G\rightarrow
\widetilde{\pi}/[\pi]_1$ and $s:G\rightarrow \widetilde{\pi}/[\pi]_1$ be two sections of the short exact sequence (1.0.2) (for $n=1$). Then the map
$$G\longrightarrow\pi^{ab},\quad g\mapsto r(g)s(g)^{-1}$$
is a $1$-cocycle taking values in the $G$-module $\pi^{ab}$ and its cohomology class only depends on the respective conjugacy classes $[r]$ and $[s]$ of $r$ and $s$. Let $[r]-[s]\in H^1(G,\pi^{ab})$ denote this cohomology class. Let us assume once more that $X$ is a general algebraic variety over $\mathbb C$ equipped with the action of the finite group $G$. Note that the cohomology group $H^1(G,\pi^{\Sigma}_1(X,x)^{ab})$ is $p$-torsion. Therefore there is a unique $\mathbb F_p$-linear map
$$s^{ab}_{X,G}:\mathbb F_p\textrm{Div}^0(\pi_0(X^G))\longrightarrow H^1(G,\pi^{\Sigma}_1(X,x)^{ab})$$
such that:
$$s^{ab}_{X,G}(y-z)=s^{1,\Sigma}_{X,G}(y)-s^{1,\Sigma}_{X,G}(z)\quad(\forall y,\forall z\in\pi_0(X^G)).$$
\end{notn}
Assume again that $X$ is a smooth, projective, geometrically irreducible curve of positive genus over $\mathbb C$ and $G$ is a group of prime order $p$ acting on $X$ non-trivially.
\begin{thm} Assume that $X^G$ is non-empty. Then the map
$$s^{ab}_{X,G}:\mathbb F_p\textrm{\rm Div}^0(X^G)\longrightarrow H^1(G,\pi^{\Sigma}_1(X,x)^{ab})$$
is surjective and its kernel is $R_{X,G}$.
\end{thm}
\begin{rem} It is immediate from the theorem above that the map $s^{1,\Sigma}_{X,G}$ is not surjective as soon as $|X^G|>4$. If $|X^G|>2$ then no element of $\textrm{\rm Div}^0(X^G)$ with support on only two points of $X^G$ can lie in $R_{X,G}$ by the first claim of Lemma 4.5, so Corollary 1.2 follows, too.
\end{rem}
\begin{proof}[Proof of Theorem 3.8] For every $z\in\textrm{Pic}^0(X)^G$ let $\overline z$ denote its connected component in $\pi_0(\textrm{Pic}^0(X)^G)$. Let $0$ denote the identity element of $\textrm{Pic}^0(X)$. There is a unique group homomorphism:
$$s^0:\pi_0(\textrm{Pic}^0(X)^G)\longrightarrow H^1(G,\pi^{\Sigma}_1(\textrm{Pic}^0(X),0)^{ab})$$
such that:
$$s^0(z)=s^{1,\Sigma}_{\textrm{Pic}^0(X),G}(z)-s^{1,\Sigma}_{\textrm{Pic}^0(X),G}(\overline 0)\quad(\forall z\in\pi_0(\textrm{Pic}^0(X)^G)).$$
The map $s^0$ is a bijection by Corollary 3.7. Therefore the commutative group $\pi_0(\textrm{Pic}^0(X)^G)$ is $p$-torsion. Let $[\cdot]:X\rightarrow\textrm{\rm Pic}^1(X)$ be the same map as in the proof of Proposition 4.4. By the above there is a unique group homomorphism:
$$\mathbb F_p\pi_0:\mathbb F_p\textrm{\rm Div}^0(X^G)\longrightarrow \pi_0(\textrm{Pic}^0(X)^G)$$
such that
$$\mathbb F_p\pi_0(z-y)=\overline{[z]-[y]}\quad(\forall y,\forall z\in X^G).$$
Let $\tau_x:\textrm{\rm Pic}^1(X)\rightarrow\textrm{\rm Pic}^0(X)$ be the map given by the rule $z\mapsto z-[x]$. Clearly the diagram
$$
\CD\mathbb F_p\textrm{\rm Div}^0(X^G)@>s^{ab}_{X,G}>>H^1(G,\pi^{\Sigma}_1(X,x)^{ab})\\
@V\mathbb F_p\pi_0VV@VV{\phi}V\\
\pi_0(\textrm{Pic}^0(X)^G)@>s^0>>
H^1(G,\pi^{\Sigma}_1(\textrm{Pic}^0(X),0)^{ab})
\endCD
$$
is commutative where $\phi$ is furnished by $\tau_x\circ[\cdot]:X\rightarrow\textrm{Pic}^0(X)$. Since $\tau_x\circ[\cdot]$ induces an isomorphism between the abelianised fundamental groups of $X$ and $\textrm{\rm Pic}^0(X)$ we get that $\phi$ is an isomorphism, too. Therefore it will be enough to show that $\mathbb F_p\pi_0$ is surjective and its kernel is $R_{X,G}$.

For every smooth projective connected curve $Y$ over $\mathbb C$ let  $\textrm{Div}^0(Y)$, $\textrm{Prin}(Y)$ denote the group of degree zero divisors on $Y$, and its subgroup of principal divisors on $Y$, respectively. Let $\textrm{Div}^0(X)^G$, $\textrm{Prin}(X)^G$ denote the group of degree zero $G$-invariant divisors on $X$ and its subgroup of principal $G$-invariant divisors on $X$, respectively. By Propositions 4.2 and 4.3 we have: $$\pi_0(\textrm{Pic}^0(X)^G)=\textrm{Div}^0(X)^G/\big(
\pi^*(\textrm{Div}^0(X/G))+\textrm{Prin}(X)^G\big).$$
Every $G$-invariant divisor $D$ of degree zero on $X$ can be written in the form:
$$\sum_{u\in X^G}n_u[u]+\sum_{z\in X-X^G}n_z(\sum_{g\in G}[g(z)]),$$
where $n_u,n_z\in\mathbb N$, all but finitely many $n_z$-s are non-zero, and $$\sum_{u\in X^G}n_u+\sum_{z\in X-X^G}pn_z=0.$$
We have $D\in\pi^*(\textrm{Div}^0(X/G))$ if and only if $p$ divides $n_u$ for every $u\in X^G$. We get that
$$\mathbb F_p\textrm{\rm Div}^0(X^G)=\textrm{Div}^0(X)^G/
\pi^*(\textrm{Div}^0(X/G)).$$
It will be enough to determine the image of $\textrm{Prin}(X)^G$ in $\mathbb F_p\textrm{\rm Div}^0(X^G)$ under this identification. Let $\beta\in F(X)^*$ be an element such that its divisor $(\beta)$ is $G$-invariant. Let $g$ be a generator of the Galois group $G$ of the extension $F(X)/F(X/G)$. Then $(\beta)=(g(\beta))$ and hence $\beta$ is an eigenvector of the $\mathbb C$-linear map $g$. Since the order of $g$ is $p$ we get that $g(\beta)=\eta\beta$ where $\eta\in\mathbb C^*$ is a $p$-th root of unity. Therefore $\beta^p\in F(X/G)$ by Galois theory. The claim now follows from Lemma 4.5.
\end{proof}

\section{The nilpotent section conjecture}

For every smooth projective irreducible curve $Y$ over $\mathbb C$, let $g_Y$ denote its genus.
\begin{prop} The map
$$s^{\infty,\Sigma}_{X,G}:X^G\longrightarrow S(\mathbf E^{\Sigma}(X,G)_{\infty})$$
is injective.
\end{prop}
\begin{proof} Clearly the map $s^{\infty,\Sigma}_{X,G}$ is injective if $s^{1,\Sigma}_{X,G}$ is. Assume that this is not the case; then $|X^G|=2$ by Corollary 1.2. Let $\pi:Y\rightarrow X$ be the unique finite \'etale Galois cover with Galois group isomorphic to $\mathbb F_p^{2g_X}$. Equip $Y$ with the $G$-action corresponding to the section of any of the two points in $X^G$. The map $\pi$ is $G$-equivariant with respect to this action. For both points $y\in X^G$ the set $\pi^{-1}(y)$ is $G$-invariant, and by assumption contains a point fixed by $G$. Moreover the cardinality of $\pi^{-1}(y)$ is $p^{2g_Y}$, and hence the number of points fixed by $G$ in this set must be divisible by $p$. We get that $2p\leq|Y^G|$. Hence $s^{1,\{p\}}_{Y,G}$ is injective and so is $s^{\infty,\Sigma}_{X,G}$.
\end{proof}
The next result is the usual pro-$\Sigma$ section conjecture for group actions on curves (Theorem 5.11 of \cite{Pa} on page 366) in the special case when $G$ has prime order. The argument in the proof of Proposition 2.5 can be used to reduce the general case to this one. 
\begin{thm} The map
$$s^{\Sigma}_{X,G}:X^G\longrightarrow S(\mathbf E^{\Sigma}(X,G))$$
is a bijection.
\end{thm}
\begin{proof} By Proposition 5.1 above the map $s^{\infty,\Sigma}_{X,G}$ is injective, therefore $s^{\Sigma}_{X,G}$ is also injective. In order to show that $s^{\Sigma}_{X,G}$ is surjective, it will be enough to show that if $X^G$ is non-empty when $S(\mathbf E^{\Sigma}(A,G))$ is non-empty, by Tamagawa's trick. If $S(\mathbf E^{\Sigma}(A,G))$ is non-empty then $S(\mathbf E^{\Sigma}(A,G)_1)$ is also non-empty. The claim now follows from Proposition 4.4.
\end{proof}
Let
$$f_{\Sigma,p}(X,G):S(\mathbf E^{\Sigma}(X,G)_{\infty})
\longrightarrow S(\mathbf E^{\{p\}}(X,G)_{\infty})$$
denote the map which assigns to the conjugacy class of a section $s$ of $\mathbf E^{\Sigma}(X,G)_{\infty}$ the conjugacy class of the composition of $s$ and the quotient map:
\begin{equation}
\pi_1^{\Sigma}(X,x)/[\pi_1^{\Sigma}(X,x)]_{\infty}\longrightarrow
\pi_1^{\{p\}}(X,x)/[\pi_1^{\{p\}}(X,x)]_{\infty}.
\end{equation}
\begin{lemma} The map $f_{\Sigma,p}(X,G)$ is a bijection.
\end{lemma}
\begin{proof} Because the group $\pi_1^{\Sigma}(X,x)/[\pi_1^{\Sigma}(X,x)]_{\infty}$ is pro-nilpotent, that is, every quotient group by an open normal subgroup is nilpotent, the $p$-Sylow subgroup in every such quotient is unique and a normal subgroup. Hence the $p$-Sylow subgroup $P$ of $\pi_1^{\Sigma}(X,x)/[\pi_1^{\Sigma}(X,x)]_{\infty}$ in the sense of Serre (see page 5 of \cite{Se}) is unique, and it is a closed normal subgroup. Since in every finite nilpotent group the unique $p$-Sylow is a direct factor, the group $P$ maps isomorphically onto $\pi_1^{\{p\}}(X,x)/[\pi_1^{\{p\}}(X,x)]_{\infty}$ under the quotient map in (5.2.1). The image of $G$ must land in $P$ with respect to every section, so the forgetful map is surjective. The injectivity is also clear, since if the image of two sections of $\mathbf E^{\Sigma}(X,G)_{\infty}$ are conjugate with respect to the map in (5.2.1), they are conjugate in $P$. 
\end{proof}
\begin{thm} The map
$$s^{\infty,\Sigma}_{X,G}:X^G\longrightarrow S(\mathbf E^{\Sigma}(X,G)_{\infty})$$
is a bijection.
\end{thm}
\begin{proof} Every finite $p$-group is nilpotent so $S(\mathbf E^{\{p\}}(X,G))=S(\mathbf E^{\{p\}}(X,G)_{\infty})$ and $s^{\{p\}}_{X,G}=s^{\infty,\{p\}}_{X,G}$. So the claim holds for $\Sigma=\{p\}$ by Theorem 5.2. Now the general case follows from Lemma 5.3.
\end{proof}

\section{On the $2$-nilpotent section conjecture}

\begin{notn} Let $s:G\rightarrow\widetilde{\pi}/[\pi]_1$ be a section of the short exact sequence $\mathbf E_1$ in (1.0.2). There is a section $\widetilde s:G\rightarrow\widetilde\pi/[\pi]_2$ of $\mathbf E_2$ such that the composition of $\widetilde s$ and the quotient map: 
$$\widetilde{\pi}/[\pi]_2\longrightarrow\widetilde{\pi}/[\pi]_1$$
is in the conjugacy class of $s$ if and only if the extension
$$
\begin{CD}
1@>>>[\pi]_1/[\pi]_2@>>>
(\pi/[\pi]_2)\times_{\pi/[\pi]_1}G
@>>>G@>>>1
\end{CD}
$$
obtained by pulling back
\begin{equation}
\CD
1@>>>[\pi]_1/[\pi]_2@>>>\pi/[\pi]_2@>>>\pi/[\pi]_1@>>>1
\endCD
\end{equation}
along $s$ splits. When $\mathbf E$ splits the extensions $\mathbf E_n$ inherit splittings, which induce bijections between $H^1(G,\pi/[\pi]_n)$ and the conjugacy classes of sections of $\mathbf E_n$. Then for any $s$ as above, there exists an $\widetilde s$ as above if and only if the class of $s$ vanishes under
$$\delta_2 : H^1(G,\pi/[\pi]_1) \rightarrow H^2(G,[\pi]_1/[\pi]_2)$$
where $\delta_2$ is the boundary map in continuous group cohomology from the extension (6.1.1).
\end{notn}
\begin{defn} Let
$$[\cdot,\cdot]:\pi/[\pi]_1\otimes\pi/[\pi]_1\longrightarrow [\pi]_1/[\pi]_2$$
be the bilinear map given by $[y,z]=\overline y\overline z
\overline y^{-1}\overline z^{-1}$ where $\overline y,\overline z\in\pi/[\pi]_2$ maps to $y,z$, respectively, with respect to the quotient map $\pi/[\pi]_2
\rightarrow\pi/[\pi]_1$. (Since the choices of the lifts $\overline y$ and $\overline z$ differ by an element of the centre, this map is well-defined.) Let
$$\vee:H^1(G,\pi/[\pi]_1)\otimes H^1(G,\pi/[\pi]_1)\longrightarrow
H^2(G, [\pi]_1/[\pi]_2)$$
be the cup product furnished by the map $[\cdot,\cdot]$.
\end{defn}
\begin{prop} For every $y,z\in H^1(G,\pi/[\pi]_1)$ we have
$$\delta_2(y+z)=\delta_2(y)+\delta_2(z)+y\vee z.$$
\end{prop}
\begin{proof} See \cite{Za}.
\end{proof}
\begin{defn} In the rest of the paper we assume that $|G|=2$. For every abelian pro-group $L$ let $L\wedge L$ denote the quotient of $L\otimes L$ by the closure of the group $\langle\{a\otimes a|a\in L\}\rangle$.  Let $L$ be a pro-$G$ module. The pairing
$$\cup:H^1(G,L)\otimes H^1(G,L)\rightarrow H^2(G, L \wedge L)$$
induced by the cup product satisfies $a\cup a=0$ for every $a\in H^1(G,L)$, and hence it furnishes a natural map
$$\wedge:H^1(G,L)\wedge H^1(G,L)\longrightarrow H^2(G, L\wedge L).$$
\end{defn}
\begin{lemma} Assume that $L$ has no $2$-torsion. Then the map
$$\wedge:H^1(G, L)\wedge H^1(G, L)\longrightarrow H^2(G,L\wedge L)$$
is injective.
\end{lemma}
\begin{proof} This is Lemma 3.3 of \cite{Wi}. 
\end{proof}
\begin{defn} Let
$$\{\cdot\}:\pi/[\pi]_1\wedge\pi/[\pi]_1\longrightarrow [\pi]_1/[\pi]_2$$
be the unique map with the property:
$$\{y\wedge z\}=[y,z],\quad(\forall y,\forall z\in\pi/[\pi]_1).$$
For every $i\in\mathbb N$ let
$$H^i(\{\cdot\}):H^i(G,\pi/[\pi]_1\wedge\pi/[\pi]_1)\longrightarrow
H^i(G,[\pi]_1/[\pi]_2)$$
induced by $\{\cdot\}$.
\end{defn}
\begin{lemma} The order of the kernel of the map:
$$H^2(\{\cdot\}):H^2(G,\pi^{\Sigma}_1(X,x)^{ab}\wedge\pi^{\Sigma}_1(X,x)^{ab})
\longrightarrow H^2(G,[\pi_1(X,x)^{\Sigma}]_1/[\pi_1(X,x)^{\Sigma}]_2)$$
is at most $2$.
\end{lemma}
\begin{proof} (Compare with Lemma 3.4 of \cite{Wi}.) Let $g=g_X$ be the genus of $X$. It is well-known that $\pi_1(X,x)$ has a presentation with generators $x_1,x_2,\ldots,x_g,y_1,y_2,\ldots,y_g$ subject to the single relation:
$$x_1y_1x_1^{-1}y_1^{-1}x_2y_2x_2^{-1}y_2^{-1}\cdots x_gy_gx_g^{-1}y_g^{-1}=1.$$
Therefore there is a short exact sequence of pro-$G$ modules:
$$\CD1\rightarrow M\longrightarrow \pi^{\Sigma}_1(X,x)^{ab}\wedge\pi^{\Sigma}_1(X,x)^{ab}
@>\{\cdot\}>>
[\pi_1(X,x)^{\Sigma}]_1/[\pi_1(X,x)^{\Sigma}]_2\rightarrow 1\endCD$$
where as a pro-group $M$ is isomorphic to $\prod_{l\in\Sigma}\mathbb Z_l$. The short exact sequence above furnishes a cohomological exact sequence:
$$
H^2(G,M)\rightarrow H^2(G,\pi^{\Sigma}_1(X,x)^{ab}\wedge\pi^{\Sigma}_1(X,x)^{ab})
\rightarrow H^2(G,[\pi_1(X,x)^{\Sigma}]_1/[\pi_1(X,x)^{\Sigma}]_2).
$$
Hence it will be enough to show that the order of $H^2(G,M)$ is at most $2$.
Since $M$ is torsion-free, there is a short exact sequence of pro-$G$ modules:
$$\CD1@>>>M@>{\cdot 2}>>M@>>>M/2M@>>>1,\endCD$$
which induces a cohomological exact sequence:
$$\CD H^1(G,M/2M)@>>> H^2(G,M)@>{\cdot2}>> H^2(G,M).\endCD$$
Since the group $H^2(G,M)$ is $2$-torsion, it will be enough to show that the order of $H^1(G,M/2M)$ is at most $2$. As a $G$-module $M/2M$ is isomorphic to the trivial $G$-module $\mathbb F_2$. The claim follows.
\end{proof}
\begin{proof}[Proof of Theorem 1.3] This argument is almost exactly the same as the proof of Theorem 3.5 of \cite{Wi}. We only need to prove the claim when $S(\mathbf E^{\Sigma}(X,G)_1)$ is non-empty. In this case the set $X^G$ is non-empty by Proposition 4.4. List the elements of $X^G$ as $x_0,x_1,\ldots,x_n,x_{n+1}$. For every $i=1,2,\ldots,n+1$ let
$$s_i=s^{1,\Sigma}_{X,G}(x_i)-s^{1,\Sigma}_{X,G}(x_0)\in H^1(G,\pi^{\Sigma}_1(X,x)^{ab}).$$
By Theorem 4.7 the elements 
$$\{s_i|\ i=1,2,\ldots,n\}$$
form a basis of the $\mathbb F_p$-vector space $H^1(G,\pi^{\Sigma}_1(X,x)^{ab})$, and
\begin{equation}
s_{n+1}=s_1+s_2+\cdots+s_n.
\end{equation}
For every $I\subseteq\{1,2,\ldots,n\}$ let
$$s_I=\sum_{i\in I}s_i.$$
By Proposition 6.3 we have:
$$\delta_2(s_I)=\sum_{i\in I}\delta_2(s_i)+
\sum_{\substack{i<j\\i,j\in I}}s_i\vee s_j.$$
Since every $s_i$ has a lift to a section of $\mathbf E^{\Sigma}(X,G)_2$ of the form $s^{2,\Sigma}_{X,G}(x_i)-s^{2,\Sigma}_{X,G}(x_0)$ we have $\delta_2(s_i)=0$ for every $i=1,\ldots,n+1$. If $|I|>1$ then the element:
$$\sum_{\substack{i<j\\i,j\in I}}s_i\wedge s_j\in H^2(G,\pi^{\Sigma}_1(X,x)^{ab}
\wedge\pi^{\Sigma}_1(X,x)^{ab})$$
is non-zero by Lemma 6.5. When $I=\{1,2,\dots,n\}$ we have $\delta_2(s_I)=0$ by (6.8.3) and hence $\sum_{\substack{i<j\\i,j\in I}}
s_i\wedge s_j$ is the unique non-zero element of the kernel of the map $H^2(\{\cdot\})$ by Lemma 6.8. Therefore we get that $\delta_2(s_I)=0$ if and only if $I$ is the empty set, it has cardinality one or $I=\{1,2,\dots,n\}$. Since every element of $H^1(G,\pi^{\Sigma}_1(X,x)^{ab})$ can be written in the form $s_I$ for some $I$ as above, the claim is now clear.
\end{proof}

\end{document}